\documentclass{article}
\usepackage{amsmath}

\def\NI{\noindent}

\newcommand{\ep}{\varepsilon}
\newcommand{\lin}{{\rm lin}\hskip0.02cm}
\newcommand{\im}{{\rm im}\hskip0.02cm}

\newcommand{\cl}{{\rm cl}\hskip0.02cm}
\newcommand{\Bbb}[1]{{\bf #1}}

\newtheorem{theorem}{Theorem}
\newtheorem{definition}{Definition}
\newtheorem{lemma}{Lemma}

\newtheorem{problem}{Problem}

\newtheorem{proposition}{Proposition}

\newenvironment{remark}[1]{\medskip\par\noindent{\bf #1.}\rm}
{\medskip\par\noindent}

\newenvironment{proof}[1]{\medskip\par\noindent{\sc #1.\ }}
{~\rule{0.5em}{0.5em}\medskip\par}

\begin{document}
\title{\LARGE{\bf{Compositions of projections in Banach spaces and relations between approximation properties}}}

\author{{\sc M.~I.~Ostrovskii}\\
\\
Department of Mathematics and Computer Science\\
St.-John's University\\
8000 Utopia Parkway\\
Queens, NY 11439, USA\\
e-mail: {\tt ostrovsm@stjohns.edu}}

\date{\today}
\maketitle

\begin{large}

\begin{remark}{Abstract} A necessary and sufficient
condition for existence of a Banach space with a finite
dimensional decomposition but without the $\pi$-property in terms
of norms of compositions of projections is found.
\end{remark}

\begin{small}
\noindent {\it 2000 Mathematics Subject Classification.} Primary
46B15; Secondary 46B07, 46B28.
\end{small}
\bigskip

The problem of existence of Banach spaces with the $\pi$-property
but without a finite dimensional decomposition is one of the
well-known open problems in Banach space theory. It was first
studied by W.~B.~Johnson \cite{Jo}. P.~G.~Casazza and N.~J.~Kalton
\cite{CK} found important connections of this problem with other
problems of Banach space theory. See in this connection the survey
\cite{casazza01}.
\medskip

Recall the definitions. A separable Banach space $X$ has the
$\pi$-{\it property} if there is a sequence $T_n:X\to X$ of finite
dimensional projections such that
$$(\forall x\in X)(\lim_{n\to\infty}||x-T_nx||=0).$$
If in addition the projections satisfy
$$(\forall n,m\in\Bbb N)(T_nT_m=T_{\min(m,n)}),$$
then $X$ has a {\it finite dimensional decomposition}.
\medskip

\begin{problem}\label{P:FDD} {\sl Does every separable
Banach space with the $\pi$-property have a finite dimensional
decomposition?}
\end{problem}

The purpose of this paper is to find an equivalent reformulation
of Problem \ref{P:FDD} in terms of norms of compositions of
projections. In the second part of the paper we discuss related
problems on compositions of projections.
\medskip

{\it Relative  projection constant} of a finite dimensional
subspace $Y$ in a normed space $X$ is defined by
$$\lambda (Y,X)=\inf\{||P||:\ P:X\to X\hbox{ is a projection
onto } Y\}.$$ In the case when $X=L_\infty(\mu)$, the constant
$\lambda(Y,X)$ is also denoted $\lambda(Y)$ (it is well known that
$\lambda(Y,L_\infty(\mu))$ depends on $Y$ only, and not on the way
in which $Y$ is embedded into $L_\infty(\mu)$).

\begin{theorem}\label{T:1} A separable
Banach space $X$ is a space with the $\pi$-pro\-per\-ty but
without a finite-dimensional decomposition if and only if there
exists an increasing sequence $\{X_i\}_{i=1}^\infty$ of
finite--dimensional subspaces of $X$ satisfying the conditions:

{\rm (a)} $\sup_i\lambda(X_i,X)<\infty$,

{\rm (b) $\cl (\cup_{i=1}^\infty X_i)=X,$

(c) }For every subsequence $\{X_{i_n}\}_{n=1}^\infty\subset
\{X_i\}_{i=1}^\infty$ and every sequence $\{P_n\}_{n=1}^\infty$ of
projections, $P_n:X_{i_{n+1}}\to X_{i_n}$, the following is true:

\begin{equation}\label{E:infty}
\sup_{k,l\in{\Bbb N},\ k<l} ||P_kP_{k+1}\dots P_{l-1}
P_l||=\infty.
\end{equation}
\end{theorem}

\begin{proof}{Proof} The ``only if'' part of the theorem is a slight modification of
Theorem 3 from W.B.~John\-son \cite{Jo}. We sketch its proof for
convenience of the reader. Let $X$ be a separable Banach space
with the $\pi$-property but without a finite dimensional
decomposition. Using the standard perturbation argument (see, for
example, \cite{JRZ}) we get that there exists an increasing
sequence $\{X_i\}_{i=1}^\infty$ of finite--dimensional subspaces
of $X$ satisfying the conditions (a) and (b). Suppose that
$\{X_i\}_{i=1}^\infty$ does not satisfy (c). Then there exists a
subsequence $\{X_{i_n}\}_{n=1}^\infty\subset \{X_i\}_{i=1}^\infty$
and a sequence $\{P_n\}$ of projections; $P_n:X_{i_{n+1}}\to
X_{i_n}$ such that
\begin{equation}\label{cont}
\sup_{k,l\in{\Bbb N},\ k<l} ||P_kP_{k+1}\dots P_{l-1}
P_l||<\infty.
\end{equation}

Let us define operators $T_k^n:X_{i_k}\to X_{i_n}$ by $T_k^nx=P_n
P_{n+1}\dots P_{k-1}x$ for $k>n,\ k,n\in\Bbb N$. Then the sequence
$\{T^n_kx\}_{k=n+1}^\infty$ is eventually constant for every
$x\in\cup_{n=1}^\infty X_{i_n}$. The inequality (\ref{cont})
implies that the sequence $\{T_k^n\}_{k=n+1}^\infty$ is uniformly
bounded. Hence it is strongly convergent. We denote its strong
limit by $T_n$. It is easy to see that $T_n$ is a continuous
projection onto $X_{i_n}$. Therefore $T_iT_j=T_j$ for $i\ge j$.
Now let $i<j$. We have
$$T_iT_jx=s-\lim_{m\to\infty}(P_i\cdot\dots\cdot P_{m-1})(T_jx)=
P_i\cdot\dots\cdot P_{j-1}(T_jx)=T_ix.$$ Hence $X$ has a finite
dimensional decomposition, contrary to the assumption.
\medskip

We turn to the ``if'' part of the theorem. We assume that $X$
contains an increasing sequence $\{X_i\}_{i=1}^\infty$ of finite
dimensional subspaces satisfying the conditions (a)-(c). It is
clear that $X$ has the $\pi$-property. In order to show that $X$
does not have a finite-dimensional decomposition, assume the
contrary. Then $X$ contains an increasing sequence
$\{Z_i\}_{i=1}^\infty$ of finite--dimensional subspaces, such that
$$\cl\left(\bigcup_{i=1}^\infty Z_i\right)=X,$$
and there exist pairwise commuting projections $T_i:X\to Z_i$ with
$\im T_i=Z_i$, for which $\sup_i||T_i||<\infty$.
\medskip

We need the following analogue of \cite[Proposition 1.a.9
(i)]{LT1} for finite-dimensional decompositions (it can be proved
using the same argument), see \cite[Section 1.g]{LT1} for
terminology related to finite dimensional decompositions.

\begin{proposition}\label{FDD} Let $\{W_i\}_{i=1}^\infty$ be a finite dimensional decomposition of $X$ with the
decomposition constant $K$. Let $E_i:W_i\to X$ be linear operators
satisfying $||E_iw-w||\le\ep_i||w||$ for each $w\in W_i$, where
$\ep_i>0$ are such that $\sum_{i=1}^\infty\ep_i<1/(2K)$. Then the
spaces $\{E_i(W_i)\}_{i=1}^\infty$ also form a finite dimensional
decomposition of $X$.
\end{proposition}

Let $U_i=(T_i-T_{i-1})X$ (we let $T_0=0$). Proposition \ref{FDD}
implies that we may assume without loss of generality that each
$U_i$ is contained in some $X_{n_i}$. Our next purpose is to show
there exist a finite dimensional decomposition $\{\tilde
U_i\}_{i=1}^\infty$ and a subsequence $\{\tilde X_i\}\subset
\{X_i\}$, such that for $\tilde Z_i=\tilde
U_1\oplus\dots\oplus\tilde U_i$ the condition
\begin{equation}\label{in}\tilde Z_i\subset\tilde X_i\subset\tilde Z_{i+1}\
(\forall i\in{\Bbb N})\end{equation} is satisfied. Our proof of
this fact uses induction and the following lemma.

\begin{lemma}\label{step} Let $\{V_i\}_{i=1}^\infty$ be a finite
dimensional decomposition of a Banach space $X$, let $H$ be a
finite dimensional subspace of $X$ satisfying $V_i\subset H$ for
$i=1,\dots,k$, and let $\ep>0$. Then there exists a blocking
$\{Y_i\}_{i=1}^\infty$ of the decomposition
$\{V_i\}_{i=1}^\infty$, such that $Y_i=V_i$ for $i=1,\dots k$,
$Y_{k+j}=V_{m+j}$ for some $m\ge k$ and all $j\ge 2$, and
$Y_{k+1}=V_{k+1}\oplus V_{k+2}\oplus \dots\oplus V_{m+1}$; and
there exists an operator $A:Y_{k+1}\to X$ satisfying the following
three conditions:
\begin{equation}\label{small}
||Ay-y||\le\ep||y||~ \forall y\in Y_{k+1},\end{equation}
\begin{equation}\label{inX} A(Y_{k+1})\subset\lin\left( (V_1\oplus\dots\oplus
V_{m+1})\bigcup H\right),
\end{equation}
\begin{equation}\label{X}
H\subset V_1\oplus V_2\oplus\dots\oplus V_k\oplus A(Y_{k+1}).
\end{equation}
\end{lemma}

\begin{proof}{Proof of Lemma \ref{step}} Let $S_i:X\to
V_1\oplus\dots\oplus V_i$ be the natural projections corresponding
to the decomposition. Let $m\in{\bf N}$ be such that $m\ge k$ and
\begin{equation}\label{d}
||S_{m+1}x-x||\le\delta||x||~\forall x\in H,\end{equation} where
$\delta>0$ is to be selected later. Let $U=S_{m+1}H$. Observe that
$S_{m+1}|_{V_1\oplus\dots\oplus V_k}$ is the identity operator,
and hence $V_1\oplus\dots\oplus V_k\subset U$. Using the standard
perturbation argument (see \cite[Proposition 5.3]{qm}) we can
estimate the projection constant of $U$ in terms of $\delta$ and
$\lambda(H,X)$ (when $\delta$ is small). Hence
$V_1\oplus\dots\oplus V_{m+1}=U\oplus C$ for some subspace $C$,
where the norms of projections onto $U$ and $C$ are estimated in
terms of $\delta$ and $\lambda(H,X)$. This fact and the estimate
(\ref{d}) allow us to claim that the operator
$A:V_1\oplus\dots\oplus V_{m+1}\to X$ defined by
$A(u+c)=S_{m+1}^{-1}(u)+c$ for $u\in U$, $c\in C$ satisfies
(\ref{small}) if $\delta>0$ is selected to be small enough. The
condition (\ref{inX}) follows immediately from the definition of
$A$. To finish the proof it remains to observe that $Ax=x$ for
$x\in V_1\oplus\dots\oplus V_k$
\end{proof}

Now we use Lemma \ref{step} to find $\{\tilde X_i\}$ and $\{\tilde
U_i\}$. In each step we shall also find a new finite dimensional
decomposition $\{U_i^j\}_{i=1}^\infty$. Let $\ep_i>0$,
$(i=2,3\dots)$ be such that $\sum_{i=2}^\infty\ep_i<1/(2K)$.
\medskip

In the first step we let $\tilde U_1=U_1$, $\tilde X_1$ be any
$X_{n_1}$ satisfying the condition $U_1\subset X_{n_1}$, and
$\{U^1_i\}_{i=1}^\infty=\{U_i\}_{i=1}^\infty$.
\medskip

In the second step we use Lemma \ref{step} with $H=\tilde X_1$,
$k=1$, $\ep=\ep_2$, and
$\{V_i\}_{i=1}^\infty=\{U^1_i\}_{i=1}^\infty$. We let
$$\{U^2_i\}_{i=1}^\infty=\left\{U^1_1,A(Y_2),U^1_{m+2},U^1_{m+3},\dots\right\}.$$
By Proposition \ref{FDD} $\{U^2_i\}_{i=1}^\infty$ is also a finite
dimensional decomposition. We let $\tilde U_2=A(Y_2)$, $\tilde
X_2$ be any $X_{n_2}$ such that $n_2>n_1$ and $\tilde U_2\subset
X_{n_2}$. Such $n_2$ exists by the condition (\ref{inX}).
\medskip

In the third step we use Lemma \ref{step} with $H=\tilde X_2$,
$k=2$, $\ep=\ep_3$, and
$\{V_i\}_{i=1}^\infty=\{U^2_i\}_{i=1}^\infty$. Re-using the
notation $A,Y_i,m$ of Lemma \ref{step} for different objects than
in the previous step, we let
$$\{U^3_i\}_{i=1}^\infty=\left\{U^2_1,U^2_2,A(Y_3),U^2_{m+2},U^2_{m+3},\dots\right\}.$$
By Proposition \ref{FDD} $\{U^3_i\}_{i=1}^\infty$ is also a finite
dimensional decomposition. Here a bit more explanation is needed.
Observe that $\{U^3_i\}_{i=1}^\infty$ is obtained from
$\{U_i\}_{i=1}^\infty$ by making two blocks and perturbing them,
one of them is perturbed no more than for $\ep_2$ (in the sense of
the inequality (\ref{small})), the other for no more than $\ep_3$,
therefore we are in a position to apply Proposition \ref{FDD}.
\smallskip

We let $\tilde U_3=A(Y_3)$, $\tilde X_3$ be any $X_{n_3}$
satisfying $n_3>n_2$ and $\tilde U_3\subset X_{n_3}$. Such $n_3$
exists by the condition (\ref{inX}).
\medskip

We continue in an obvious way. The fact that the condition
(\ref{in}) is satisfied is clear from the construction (see the
condition (\ref{X}) in Lemma \ref{step}). It remains to check that
$\{\tilde U_i\}_{i=1}^\infty$ form a finite dimensional
decomposition of $X$. To see this observe that $\tilde U_i$ are
$\ep_i$-perturbations of a blocking of $\{U_i\}_{i=1}^\infty$.
Recalling the choice of $\ep_i$ and using Proposition \ref{FDD},
we get the desired statement.
\medskip

Let $Q_n:X\to\tilde X_n$ be some projections with
$\sup_n||Q_n||<\infty$ and $\im Q_n=\tilde X_n$. Let $R_n:X\to
\tilde Z_n$ be projections corresponding to the decomposition
$\{\tilde U_i\}_{i=1}^\infty$. We introduce new projections
$P_n:X\to \tilde X_n$ with $\im P_n=\tilde X_n$ as:
$$P_n=R_n+(I-R_n)Q_n(R_{n+1}-R_n).$$

Let us show that $P_n$ are projections onto $\tilde X_n$ and
$P_nP_{n+1}=P_n$.
\medskip

If $x\in\tilde X_n$, then $x=R_{n+1}x=R_nx+(R_{n+1}-R_n)x$. Since
$(R_{n+1}-R_n)x\in\tilde X_n$, then
$Q_n(R_{n+1}-R_n)x=(R_{n+1}-R_n)x$. Hence
$$x=R_nx+(I-R_n)Q_n(R_{n+1}-R_n)x.$$

Let us show that $\im P_n\subset\tilde X_n$. The condition
(\ref{in}) implies that $\im R_n\subset\tilde X_n$. Therefore
$(I-R_n)\tilde X_n\subset\tilde X_n$, and $P_n$ is a projection
onto $\tilde X_n$.
\medskip

Let us show that $P_nP_{n+1}=P_n$. In fact,
$$P_nP_{n+1}=(R_n+(I-R_n)Q_n(R_{n+1}-R_n))
\left(R_{n+1}+(I-R_{n+1})Q_{n+1}(R_{n+2}-R_{n+1})\right)=$$
$$R_n+(I-R_n)Q_n(R_{n+1}-R_n)=P_n.$$

It follows that $\{P_n\}$ is a uniformly bounded commuting
sequence of projections onto $\{\tilde X_n\}$. We get a
contradiction with the condition (\ref{E:infty}).
\end{proof}

Theorem \ref{T:1} shows that one of the natural approaches to
Problem \ref{P:FDD} is to start with the following problem on
composition of projections. A projection of a Banach space $X$
onto its subspace $Y$ is called {\it minimal}, if its norm is
equal to $\lambda(Y,X)$, and {\it close-to-minimal}, if its norm
is close to $\lambda(Y,X)$.
\medskip

Consider a triple $(X_1,X_2,X_3)$ of Banach spaces satisfying
$X_1\subset X_2\subset X_3$.
Assume that $X_1$ and $X_2$ are finite dimensional.
\medskip

\begin{problem} {\sl Is it possible to find a close-to-minimal
projection $P:X_3\to X_1$ which can be factored
as $P=P_1P_2$, where $P_2:X_3\to X_2$ is a close-to-minimal
projection onto $X_2$ and $P_1:X_2\to X_1$ is $P|_{X_2}$?}
\end{problem}

Some related observations.

\begin{proposition}\label{P1}
Each projection $P:X_3\to X_1$ has a factorization of the form
$P=P_1P_2$, where  $P_2:X_3\to X_2$ and $P_1:X_2\to X_1$ are
projections.
\end{proposition}

In fact, let $\ker P_1=\ker P\cap X_2$. Let $\ker P_2$ be a
complement of $\ker P_1$ in $\ker P$ (such complement exists
because $\ker P_1$ is finite dimensional).

\begin{proposition}\label{example} There exist
triples $(X_1,X_2,X_3)$ and minimal projections $P:X_3\to X_1$
which cannot be factored as $P_1P_2$, where $P_2$ is a minimal
projection onto $X_2$.
\end{proposition}

In the proof of this result and in further discussion it is
convenient to use the notion of a sufficient enlargement. We
denote the ball of a Banach space $X$ by $B_X$, in the case when
$X=\ell_p^n$, we use the notation $B_p^n$.

\begin{definition} {\rm A bounded, closed, convex, $0$-symmetric set $A$ in a finite dimensional
normed space $X$ is called a {\it sufficient enlargement} for $X$
(or of $B_X$) if for arbitrary isometric embedding $X\subset Y$
($Y$ is a Banach space) there exists a projection $P:Y\to X$ such
that $P(B_Y)\subset A$. A {\it minimal sufficient enlargement} is
defined to be a sufficient enlargement no proper subset of which
is a sufficient enlargement.}
\end{definition}

It is easy to see that if $X$ is a subspace of $L_\infty(\mu)$ and
$P:L_\infty(\mu)\to X$ is a projection, then
$\cl(P(B_{L_\infty(\mu)}))$ is a sufficient enlargement of $B_X$.
See \cite{extracta}, \cite{archiv}, and \cite{mpcps} for results on sufficient enlargements.

\begin{proof}{Proof of Proposition \ref{example}} Consider a triple of the form $\ell^k_2\subset
\ell^n_2\subset L_\infty(\mu).$ The set $\lambda(\ell_2^n)B_2^n$
is a minimal sufficient enlargement of $\ell_2^n$ (see
\cite[Section 3]{archiv}). Therefore, if $P_2: L_\infty(\mu)\to
\ell^n_2$ is a minimal projection, then $\cl
(P_2(B_{L_{\infty}(\mu)}))=\lambda(\ell^n_2)B^n_2$. Hence, for an
arbitrary $P_1:\ell_2^n\to \ell_2^k$ we have
$\cl(P_1P_2(B_{L_{\infty}(\mu)}))=\cl(P_1(\lambda(\ell^n_2)B^n_2))\supset
\lambda(\ell^n_2)B^k_2,$ where we have an equality instead of an
inclusion if $P_1$ is orthogonal.
\medskip

Of course, if $k$ is much less than $n,$ then $\lambda(\ell^k_2)$
is much less than $\lambda(\ell^n_2)$, and the projection $P_1P_2$
is far from being minimal.
\end{proof}

On the other hand, there exist $P_1:\ell_2^n\to\ell_2^k$ and
$P_2:L_\infty(\mu)\to\ell_2^n$, such that $P_1P_2$ is a minimal
projection and $P_2$ is a close-to-minimal projection. To show
this we need the following observation about sufficient
enlargements.

\begin{lemma}\label{L1} Let $X$ and $Y$ be two finite dimensional
normed spaces and $X\oplus Y$ be their direct sum.
\medskip

Suppose that $X\oplus Y$ is endowed with a norm $||\cdot||$
satisfying the conditions
$$\begin{array}{cr}
\ & ||x||\leq||(x,y)||,~\forall(x,y)\in X\oplus Y\hskip2.3cm (1)  \\
\hbox{ and } & \  \\
\ & ||y||\leq||(x,y)||,~\forall(x,y)\in X\oplus Y.\hskip2cm (2) \\
\end{array}$$

\NI Let $A_X$ be a sufficient enlargement of $B_X$ and $A_Y$ be a
sufficient enlargement of $B_Y$. Then the Minkowski sum $A_X+A_Y$
is a sufficient enlargement for $(X\oplus Y,~||\cdot||)$.
\end{lemma}

\begin{proof}{Proof} Let $X\oplus Y\subset Z$ be an isometric embedding. We show that there
exists a projection $P_X:Z\to X$ such that $P_X(B_Z)\subset A_X$
and $P_X(Y)=\{0\}$. Let $\varphi_Y:Z\to Z/Y$ be a quotient mapping
with $\ker\varphi_Y=Y$. By the condition (1) the restriction
$\varphi_Y|_X$ is an isometry. Hence, there is a projection
$Q_X:Z/Y\to \varphi_Y(X)$ such that $Q_X(B_{Z/Y})\subset
\varphi_Y(A_X)$. Therefore we may identify $X$ with $\varphi_YX$
and $A_X$ with $\varphi_Y(A_X)$. We let $P_X=Q_X\varphi_Y$. It is
clear that all of the conditions are satisfied.
\medskip

In the same way, the condition (2) implies that there exists a
projection $P_Y:Z\to Y$ such that $P_Y(B_Z)\subset A_Y$ and
$P_Y(X)=0$.
\medskip

Let $P:Z\to X\oplus Y$ be defined by $Pz=(P_X z,P_Y z)$.
\medskip

It is easy to check that $P$ is a projection onto $X\oplus Y$. In
fact,
$$P(x,y)=(P_X(x,y),P_Y(x,y))=(x,y).$$
Also $P(B_Z)\subset P_X(B_Z)+P_Y(B_Z)\subset A_X+A_Y.$
\end{proof}

Now we are ready to construct projections $P_1$ and $P_2$ whose existence was claimed
before Lemma \ref{L1}. By Lemma \ref{L1} the set
$$A=\lambda(\ell{^k_2})B{^k_2}+\lambda(\ell{^{n-k}_2})B{^{n-k}_2}$$
is a sufficient enlargement for
$\ell_2^n=\ell_2^k\oplus\ell_2^{n-k}$. Let $P_2:L_\infty(\mu)\to
\ell_2^n$ be a projection corresponding to this sufficient
enlargement, that is, satisfying $P_2(B_{L_{\infty(\mu)}})\subset
A$. It is easy to see that the norm of this projection is $\le
((\lambda(\ell{^k_2}))^2+(\lambda(\ell{^{n-k}_2}))^2)^{1/2}$.
Hence it is not much more than $\lambda(\ell{^{n}_2})$. In fact,
$((\lambda(\ell{^k_2}))^2+(\lambda(\ell{^{n-k}_2}))^2)^{1/2}<\sqrt{2}
\lambda(\ell{^{n}_{2}})$.
\medskip

\begin{remark}{Remark} By \cite[Theorem 5]{archiv} the sufficient
enlargement
$A=\lambda(\ell{^k_2})B{^k_2}+\lambda(\ell{^{n-k}_2})B{^{n-k}_2}$
is minimal. Hence $\cl(P_2(B_{L_{\infty(\mu)}}))=A$ and $||P_2||=
((\lambda(\ell{^k_2}))^2+(\lambda(\ell{^{n-k}_2}))^2)^{1/2}$.
\end{remark}

Is it always like this? More precisely

\begin{problem}\label{P:C1} {\sl
Does there exists a universal constant $C\in [1,\infty)$ such that
for each triple $X_1\subset X_2\subset X_3$ of Banach spaces, with
$X_1$ and $X_2$ finite dimensional, there exist projections
$P_1:X_2\to X_1$ and $P_2:X_3\to X_2$, such that $||P_2||\le
C\lambda(X_2,X_3)$ and $||P_1P_2||=\lambda(X_1,X_3)$?}
\end{problem}

Another version of this problem (which will be particularly interesting if Problem \ref{P:C1} has a negative answer):
\medskip

\begin{problem}\label{P:C2} {\sl
Do there exist universal constants $C_1,C_2\in [1,\infty)$ such that for each triple $X_1\subset X_2\subset X_3$ of Banach spaces, with $X_1$ and $X_2$ finite dimensional, there exist projections $P_1:X_2\to X_1$ and $P_2:X_3\to X_2$, such that $||P_2||\le C_1\lambda(X_2,X_3)$ and
$||P_1P_2||=C_2\lambda(X_1,X_3)$?}
\end{problem}

\end{large}


\begin{thebibliography}{2}
\begin{small}

\bibitem{casazza01} P.~G.~Casazza, Approximation properties,
in: {\it ``Handbook of the Geometry of Banach Spaces''}, Volume
{\bf 1}, Edited by W. B. Johnson and J. Lindenstrauss,
North-Holland Publishing Co., 2001, pp.~273--316.

\bibitem{CK} P.~G.~Casazza and N.~J.~Kalton,
Notes on approximation properties in se\-pa\-rable Banach spaces, in:
{\it Geometry of Banach Spaces}, edited by P.~M\"uller and W.~Schachermayer,
Cambridge, Cambridge University Press, 1990, pp.~49--63.

\bibitem{Jo} W.~B.~Johnson, Finite-dimensional Schauder decompositions
in $\pi_\lambda$ and dual $\pi_\lambda$ spaces, {\it Illinois J. Math.},
{\bf 14} (1970), 642--647.

\bibitem{JRZ} W.~B.~Johnson, H.~P.~Rosenthal, and
M.~Zippin, On bases, finite dimensional decompositions and weaker structures in Banach spaces,
{\it Israel J. Math.}, {\bf 9} (1971), 488--506.

\bibitem{LT1} J.~Lindenstrauss and L.~Tzafriri, {\it Classical Banach
Spaces}, v.~{\bf I}, Berlin, Springer--Verlag, 1977.

\bibitem{qm}  M.~I.~Ostrovskii, Topologies on the set of all subspaces of a Banach space
and related questions of Banach space geometry, {\it Quaestiones
Math.} {\bf 17} (1994), 259--319.

\bibitem{extracta} M.~I.~Ostrovskii, Generalization of projection constants:
sufficient enlargements, {\it Extracta Math.} {\bf 11} (1996), 466--474.

\bibitem{archiv} M.~I.~Ostrovskii , Projections in normed
linear spaces and sufficient enlargements, {\it Archiv der Mathematik},
{\bf 71} (1998), no. 4, 315--324.

\bibitem{mpcps} M.~I.~Ostrovskii, Sufficient enlargements of minimal volume for
two-dimensional normed spaces, {\it Math. Proc. Cambridge Phil.
Soc.}, {\bf 137} (2004), 377-396.

\end{small}
\end{thebibliography}
\end{document}